\documentclass[a4paper,DIV=classic,10pt]{myart}
\usepackage[normalem]{ulem}

\newcommand{\abs}[1]{\left\vert#1\right\vert}

\newcommand{\R}{\mathbb{R}}
\newcommand{\N}{\mathbb{N}}

\newcommand*{\cadlag}{c\`adl\`ag}

\newcommand{\brak}[1]{\left(#1\right)}    
\newcommand{\crl}[1]{\left\{#1\right\}}   

\DeclareMathOperator*{\essinf}{ess\,inf}

\begin{document}
\title{Stability and Markov Property of Forward Backward Minimal Supersolutions}
\author[a,1,t1]{Samuel Drapeau}
\author[b,2]{Christoph Mainberger}

\address[a]{School of Mathematical Sciences \& Shanghai Advanced Institute for Finance (CAFR/CMAR), Shanghai Jiao Tong University, 211 West Huaihai Road, Shanghai, 200030 China}
\address[b]{Technische Universit\"at Berlin, Stra\ss e des 17.~Juni 136, 10623 Berlin, Germany}

\eMail[1]{sdrapeau@saif.sjtu.edu.cn}
\eMail[2]{mainberg@math.tu-berlin.de}

\thanksColleagues{The authors would like to thank Cosima Schroller, Michael Kupper and Mete Soner as well as one anonymous referee for helpful comments and numerous fruitful discussions.}

\myThanks[t1]{Financial support from the National Science Foundation of China, ``Research Fund for International Young Scientists'', Grant number 11550110184, is gratefully acknowledged.}


\ArXiV{1503.00240}


\abstract{ 
We show stability and locality of the minimal supersolution of a forward backward stochastic differential equation with respect to the underlying forward process under weak assumptions on the generator.
The forward process appears both in the generator and the terminal condition.
Painlev\'e-Kuratowski and Convex Epi-convergence are used to establish the stability.
For Markovian forward processes the minimal supersolution is shown to have the Markov property.
Furthermore, it is related to a time-shifted problem and identified as the unique minimal viscosity supersolution of a corresponding PDE.
}
\keyWords{Supersolutions of Backward Stochastic Differential Equations; Stability; FBSDEs; Markov Property; Viscosity Supersolutions}
\keyAMSClassification{60H10, 60H30, 35D40}

\maketitle

\section{Introduction}
\addcontentsline{toc}{section}{Introduction}
\markboth{\uppercase{Introduction}}{\uppercase{Introduction}}

In this work we study \emph{forward backward minimal supersolutions}, particularly their stability and locality with respect to the forward process.
For the special case of Markovian forward processes, we thereby provide the Markov property of the minimal supersolution and show how the latter is related to viscosity supersolutions of a corresponding PDE.
More precisely, given a fixed time horizon, $T>0$, measurable functions $g$ and $\varphi$, a filtered probability space, the filtration of which is generated by a $d$-dimensional Brownian motion, and a progressive $d$-dimensional forward process $X$, we study the minimal supersolution of the decoupled \emph{forward backward stochastic differential equation} (FBSDE)
\begin{equation}\label{eq_intro01}
    Y_s - \int_s^t g_u(X_u,Y_u,Z_u)du + \int_s^t Z_u dW_u \ge Y_t \quad\text{and}\quad Y_T \geq \varphi(X_T), \tag{$\ast$}
\end{equation}
where $0\le s\le t\le T$.
Throughout we work with a \emph{standard generator} $g$, that is a positive, lower semicontinous function which is convex in the control variable $z$ and which in addition is either monotone in $y$ or jointly convex in $(y,z)$. 
The expression ``standard'' is justified since the former are, to the best of our knowledge, the mildest assumptions guaranteeing existence and uniqueness of the minimal supersolution $(\mathcal E(X),Z)$ of \eqref{eq_intro01}, compare \citet{CSTP}.

The first novel and main contribution of this paper consists in proving stability of the minimal supersolution as a function of $X$ by combining existing stability results of \citet{CSTP} and \citet{henner2013} with Painlev\'e-Kuratowski and Convex epigraphical convergence.
This kind of stability generalizes results obtained so far in this direction in that the forward process now affects jointly both the dynamics of the problem through its input on $g$ and the terminal condition.
It comes at a cost in terms of assumptions on the generator, namely at the need of $g$ satisfying not only a point-wise but also an epigraphical lower semi-continuity condition \ref{rec}.
However, we show that this epigraphical lower semi-continuity condition is met in a significant number of situations using some results about horizon functions, compare \citet{rockafellar02}, and Paintlev\'e-Kuratovsky/Convex epigraphical convergence in \citet{aubin2009,loehne2006}.
Furthermore, we prove that the minimal supersolution is local in the following sense: Given a time $t\in[0,T]$ and a set $A\in\mathcal F_t$ it holds $\mathcal{E}_s(X)=1_{A}\mathcal{E}_s(X^1)+1_{A^c}\mathcal{E}_s(X^2)$ for $s\in[t,T]$ where $X^1$ and $X^2$ are two forward processes and $X$ their concatenation.
Specifically, this allows to restrict our focus to supersolutions on $[t,T]$ and forget about the past once we have arrived at time $t$.

Both the results above open the door to the study of supersolutions of Markovian FBSDEs and of their relation to PDE theory, the second part of this work. 
Supposing $X$ to be the solution to a classical SDE we study under which conditions $\mathcal E$ is also Markovian in the sense of it being a function of time and the underlying forward process.
To this end, we shift the original problem \eqref{eq_intro01} in time and introduce the candidate function $u(t,x)$, the value at time zero of the minimal supersolution corresponding to the shifted formulation with a forward process starting in $x\in\R^d$.
Besides proving that $x\mapsto u(t,x)$ maintains central features such as lower semicontinuity, we show that $\mathcal E_t(X^{t,x})=u(t,x)$ where $X^{t,x}$ is the forward diffusion starting in $x$ at time $t$, therewith establishing the connection between the original and the time-shifted problem.
Furthermore, using $X=X^{t,X_t}$ and approximating $X_t$ from below by step functions, we obtain that $\mathcal E_t(X)\geq u(t,X_t)$ always holds true, with equality if $x\mapsto u(t,x)$ is monotone or continuous.

For $\varphi$ bounded from below and $g$ jointly convex in $(x,y,z)$  
another ansatz to obtain the desired representation $\mathcal E_t(X)=u(t,X_t)$ is to draw on both the convexity of the generator
and the relation of Lipschitz BSDEs and PDEs as for instance given in \citet{karoui01}.
The former allows to approximate $g$ from below by a sequence of Lipschitz generators for which the minimal supersolution coincides with the unique solution of the BSDE, a method first used in \citet{DrapeauTangpi}.
The latter in turn then ensures that at each approximation step there is a a one-to-one relation between the (super-)solution and a viscosity solution of the corresponding PDE.
Stability of the problem with respect to the generator, compare \citet{CSTP}, finally allows us to pass to the limit and thereby identify $u$ as the unique minimal lower semi-continuous viscosity supersolution of the above PDE.
This extends existing results on the connection of BSDEs and PDEs to minimal supersolutions and constitutes the third contribution of this work.

Let us briefly discuss the existing literature on related problems.
Nonlinear BSDEs were first introduced in \citet{peng01}, whereas their relation to PDEs was extensively
studied among other in \citet{PadouxPeng1992} and \citet{Peng1992}.
As BSDEs may be ill-posed beyond the quadratic case, compare \citet{DelbaenHu}, minimal supersolutions extend the concept of solutions and were first rigorously studied in \citet{CSTP} and then subsequently in \citet{HKM}, while \citet{DrapeauTangpi} derived their dual representation.
In order to keep the presentation neat, we refer the reader to aforementioned works and \citet{karoui01} for a broader discussion on the subject.


The remainder of this paper is organized as follows. 
Setting and notations are specified in Section \ref{sec2}, while the central results on stability and locality are given in Section \ref{sec3}.
Subsequently, Section \ref{sec4} covers the study of the Markovian case,
whereas the relation between forward backward minimal supersolutions and viscosity supersolutions of PDEs is provided in Section \ref{sec5}.
Technical results on epi-convergenge and Painlev\'e-Kuratowski limits are presented in the appendix.


\section{Setting and notation}\label{sec2}

We consider the canonical probability space $(\Omega,\mathcal F)=(C_0([0,T],\R^d),\mathcal B(C_0([0,T],\R^d)))$.
By $W$ we denote the canonical process, $P$ the Wiener measure and $(\mathcal F_t)$ the filtration generated by $W$ augmented by the $P$-null sets of $W$.
For some fixed time horizon $T>0$ the set of $\mathcal F_T$-measurable random variables is denoted by $L^0$, where random variables are identified in the $P$-almost sure sense.
Let furthermore denote $L^p$ the set of random variables in $L^0$ with finite $p$-norm, for $p \in [ 1,+\infty]$.
Inequalities and strict inequalities between any two random variables or processes $X^1,X^2$ are understood in the $P$-almost sure or in the $P\otimes dt$-almost everywhere sense, respectively.
We denote by $\mathcal{S}$ the set of \cadlag\, progressively measurable processes $Y$ with values in $\R$.
We further denote by $\mathcal{L}$ the set of $\R^{d}$-valued, progressively measurable processes $Z$ such that $\int_0^T Z_s^2 ds <\infty$ $P$-almost surely.
For $Z\in \mathcal L$, the stochastic integral $\int ZdW$ is well defined and is a continuous local martingale.

We define the concatenation of $\bar\omega, \omega\in \Omega$ at time $t\in[0,T]$ by
\begin{equation}\label{eq_omega_concatenation}
    (\bar\omega\otimes_t \omega)_u := \bar\omega_u 1_{[0,t)}(u) + 
    \big(\bar\omega_{t} + \omega_{u}-\omega_{t}\big) 1_{[t, T]}(u),\quad u \in [0,T].
\end{equation}

Given an extended real valued function $(x,y,z)\mapsto g(x,y,z)$ defined on a finite dimensional space, we denote $\text{dom}g=\{(x,y,z)\colon g(x,y,z)<\infty\}$ and, by a slight abuse of notation, we say that $x \in \text{dom}g$ if $g(x,y,z)<\infty$ for some $y,z$.
Further, for a sequence $(x_n)\subseteq \mathbb{R}^d$ we denote by $\text{cl}\{g(x_n,\cdot,\cdot)\colon n\}$  the greatest lower semi-continuous function $(y,z)\mapsto h(y,z)$ such that $h\leq g(x_n,\cdot,\cdot)$ for every $n$, while $\text{clco}\{g(x_n,\cdot,\cdot)\colon n\}$ or $\text{clco}_z \{g(x_n,\cdot,\cdot)\colon n\}$ is defined likewise with the addition of being jointly convex or convex in $z$, respectively.
This given, we define the Painlev\'e-Kuratowski and Closed-Convex limit inferior as follows, see Appendix \ref{appendix01},
\begin{equation}\label{app_defn_limit}
    \begin{split}
        \mathfrak{e}\text{-}\liminf g(x_n,\cdot,\cdot)&:= \sup_n \text{cl}\{g(x_k,\cdot,\cdot)\colon k\geq n\}\\
        \mathfrak{c}\text{-}\liminf g(x_n,\cdot,\cdot)&:=\sup_n \text{clco}\{g(x_k,\cdot,\cdot)\colon k\geq n\}\\
        \mathfrak{c}_z\text{-}\liminf g(x_n,\cdot,\cdot)&:=\sup_n \text{clco}_z\{g(x_k,\cdot,\cdot)\colon k\geq n\}.
    \end{split}
\end{equation}

Finally, for a lower semi-continuous proper convex function $h$, we denote by $h^\infty$ the \emph{horizon function} of $h$, that is,
\begin{equation*}
    h^\infty(y)=\lim_{\alpha \to \infty } \frac{h(x+\alpha y)-h(x)}{\alpha}, 
\end{equation*}
where $x\in \text{dom}f$, \citep[See][Definition 3.17 and Theorem 3.21]{rockafellar02}.


\section{Forward backward minimal supersolutions}\label{sec3}
Throughout we call a jointly measurable function $g\colon [0,T] \times \R^d \times \R\times \R^{d} \to [-\infty, \infty]$ a \emph{generator}.
Given a generator $g$, a progressive $d$-dimensional measurable process $X$ and a measurable function $\varphi :\R^d\to \mathbb{R}$ we call a pair $( Y, Z)\in \mathcal{S}\times\mathcal{L}$ a \emph{supersolution of the decoupled forward backward stochastic differential equation}\footnote{To keep the presentation lean we sometimes use the abbreviated expression \emph{forward backward supersolutions.}} if
\begin{equation}
     Y_s - \int_s^t g_u(X_u,Y_u,Z_u)du + \int_s^t Z_u dW_u \geq Y_t \quad \text{and}\quad Y_T \geq \varphi(X_T)
     \label{eq_original_tarpo}
\end{equation}
for every $0\leq s\leq t\leq T$.
We call $X$ the \emph{forward process}, $Y$ the \emph{value process} and $Z$ its corresponding \emph{control process}.
A control process $Z \in \mathcal{L}$ is said to be \emph{admissible} if the continuous local martingale $\int Z dW$ is a supermartingale and we denote the set collecting all supersolutions by
\begin{equation*}
    \mathcal{A}(X):=\crl{(Y,Z) \in \mathcal{S}\times \mathcal L \colon Z\text{ is admissible and \eqref{eq_original_tarpo} holds}}.
\end{equation*}
In general supersolutions are not unique, therefore we define a supersolution $(Y,Z)\in \mathcal{A}(X)$ to be \emph{minimal} if $Y\leq \hat{Y}$ for every $(\hat{Y},\hat{Z})\in \mathcal{A}(X)$.
If a minimal supersolution exists, we denote its value process by $\mathcal{E}(X)$.
If further, $\mathcal{A}(X)\equiv\emptyset$, we set $\mathcal{E}(X)=\infty$ by convention.

Throughout this paper a generator may satisfy
\begin{enumerate}[label=\textsc{(std)},leftmargin=40pt]
    \item  $g$ is positive, lower semicontinuous and $z \mapsto g(x,y,z)$ is convex.\label{std}
\end{enumerate}
\begin{enumerate}[label=\textsc{(mon)},leftmargin=40pt]
    \item $y \mapsto g(x,y,z)$ is monotone.\footnote{That is either increasing or decreasing.}\label{mon}
\end{enumerate}
\begin{enumerate}[label=\textsc{(con)},leftmargin=40pt]
    \item $(y,z)\mapsto g(x,y,z)$ is jointly convex.\label{con1}
\end{enumerate}
\begin{definition}
    We say that $g$ is a \emph{standard generator} if $g$ satisfies \ref{std} and either \ref{mon} or \ref{con1}.
\end{definition}
\begin{remark}
    Following \citep[Section 4.3]{CSTP}, the positivity assumption in \ref{std} may be relaxed to $g$ being bounded from below by an affine function of $z$ without violating the validity of our results.
\end{remark}
The following is a straightforward application of results in \citep{CSTP,DrapeauTangpi}.
\begin{theorem}\label{thm:existunique}
    Let $g$ be a standard generator.
    Suppose that $\varphi(X_T)^- \in L^1$ and $\mathcal{A}(X)$ is non-empty.
    Then there exists a unique minimal supersolution $(\mathcal{E}(X),Z) \in \mathcal{A}(X)$ for which holds
    \begin{equation*}
        \mathcal{E}_t(X)=\essinf \left\{ Y_t \colon (Y,Z)\in \mathcal{A}(X) \right\}
    \end{equation*}
    almost surely for every $0\leq t\leq T$.
\end{theorem}
\begin{proof}
    For a given $X \in \mathcal{S}$, setting $g^X(y,z):=g(X,y,z)$ and $\xi =\varphi(X_T)$ defines a generator and a terminal condition satisfying the existence and uniqueness assumptions in \citep{CSTP,DrapeauTangpi}, hence the assertion.
\end{proof}
Denoting by $\mathcal{A}(\xi, h)$ and $\mathcal{E}(\xi,h)$ the set of supersolutions and the minimal supersolution, respectively, with terminal condition $\xi$ and generator $h(y,z)$ in the sense of \citep{CSTP,DrapeauTangpi}, it holds $\mathcal{E}(X)=\mathcal{E}(\varphi(X_T),h)$ where $h=g(X,\cdot,\cdot)$.

The subsequent results of Sections \ref{sec4}  and \ref{sec5} depend on the stability of the minimal supersolution as a function of $X$,
provided in Theorem \ref{prop:stability} below.
Together with the subsequent Proposition \ref{prop:lsc}, it constitutes the first main contribution of this work, generalizes the stability results given in \citep{CSTP} and is partially inspired by driver stability shown in \citep{henner2013}.
However, by dependence of the generator on the forward component we obtain a joint stability in the driver and terminal condition.
This requires a novel approach and one further assumption on the generator.
\begin{enumerate}[label=\textsc{(rec)},leftmargin=40pt]
    \item \label{rec} for every bounded sequence $(x_n)$ such that $x_n\to x$, it holds
        \begin{itemize}
            \item if $g$ satisfies \ref{con1}, then $g(x,\cdot,\cdot)\leq \mathfrak{c}\text{-}\liminf g(x_n,\cdot,\cdot)$;
            \item if $g$ satisfies \ref{mon}, then $g(x,\cdot,\cdot)\leq \mathfrak{c}_{z}\text{-}\liminf g(x_n,\cdot,\cdot)$.
        \end{itemize}
\end{enumerate}
\begin{theorem}\label{prop:stability}
    Let $g$ be a standard generator satisfying \ref{rec} and suppose that $\varphi$ is lower semicontinuous.
    Let $(X^n)$ be a sequence of progressive measurable processes such that $X^n_t\to X_t$ almost surely for every $t$ and $\varphi(X^n_T) \geq -\eta$ where $\eta \in L^1_+$.
    Then it holds
    \begin{equation}\label{thm_stab1}
        \mathcal{E}_0(X)\leq \liminf \mathcal{E}_0(X^n).
    \end{equation}
    If furthermore $x \mapsto g(x,\cdot,\cdot)$, $\varphi$ and $(X^n)$ are increasing, then
    \begin{equation}\label{thm_stab2}
        \mathcal{E}_0(X)=\lim \mathcal{E}_0(X^n).
    \end{equation}
    Finally, if $\liminf \mathcal{E}_0(X^n)<\infty$, then $\mathcal{E}_t(X)\leq \liminf \mathcal{E}_t(X^n)$ for every $t$.
\end{theorem}
\begin{proof}
    We define\footnote{The following operations are to be understood $(t,\omega)$-wise.}
    \begin{itemize}
        \item if $g$ satisfies \ref{mon}: $h^n :=\text{clco}_z\{g(X^k,\cdot,\cdot);k\geq n\}$ for which holds that $h^n$ is positive, lower semicontinuous, monotone in $y$ and convex in $z$.
            Furthermore, it holds $h^n\leq h^{n+1}$ and $h^n\to h=\mathfrak{c}_z\text{-}\liminf g(X^n,\cdot,\cdot)$ by definition of $\mathfrak{c}_z\text{-}\liminf$ in \eqref{app_defn_limit}.
        \item if $g$ satisfies \ref{con1}: $h^n :=\text{clco}\{g(X^k,\cdot,\cdot);k\geq n\}$ for which holds that $h^n$ is positive, lower semicontinuous, and jointly convex in $(y,z)$.
            Furthermore, it holds $h^n\leq h^{n+1}$ and $h^n\to h=\mathfrak{c}\text{-}\liminf g(X^n,\cdot,\cdot)$ by definition of $\mathfrak{c}\text{-}\liminf$ in \eqref{app_defn_limit}.
    \end{itemize}
    Define in addition the increasing sequence of terminal conditions $\xi^n=\inf_{k\geq n} \varphi(X^k_T)$ for which holds $\xi^n\geq -\eta$ for every $n$ and $\xi:=\sup \xi^n$. 

    Given the sequences of terminal conditions $(\xi^n)$ and generators $(h^n)$, both increasing, we adapt the stability proofs in \citep{CSTP} as follows.
    The monotonicity of the minimal supersolution operator implies $\mathcal{E}_0(\xi^1, h^1)\leq \ldots \leq \mathcal{E}_0(\xi^n, h^n)\leq \ldots \leq \mathcal{E}_0(\xi, h)$.
    Let us show that $\lim \mathcal{E}_0(\xi^n, h^n)=\mathcal{E}_0(\xi, h)$.
    If $\lim \mathcal{E}_0(\xi^n,h^n)= \infty$, there is nothing to prove.
    Assuming therefore that $\lim \mathcal{E}_0(\xi^n, h^n)<\infty$ yields the existence of 
    a non-trivial minimal supersolution for every $n$ .
    Denote by $((Y^n,Z^n))$ this sequence of minimal supersolutions and define $Y=\lim Y^n$ since $(Y^n)$ is increasing.
    The same argumentation as in \citep{CSTP} implies $Y$ being a \cadlag\, supermartingale and the existence of $Z \in \mathcal{L}$ together with a sequence $(\tilde{Z}^n)$ in the asymptotic convex hull of $(Z^n)$ such that $\tilde{Z}^n\to Z$ $P\otimes dt$-almost surely, while $\int \tilde{Z}^n dW\to \int ZdW$ locally in $L^1$.
    Further, $\int Z dW$ is a admissible.
    We are left to show that $(Y,Z)$ is a minimal supersolution for $\xi$.
    Since $h^k(Y,Z)\to h(Y,Z)$ $P\otimes dt$-almost surely, Fatou's Lemma yields
    \begin{equation}\label{eq:temp01}
        Y_s-\int_s^th_u(Y_u,Z_u)du+\int_s^t Z_u dW_u\geq \limsup_k \left(Y_s-\int_s^th_u^k(Y_u,Z_u)du+\int_s^t Z_u dW_u\right).
    \end{equation}
    For $k$ fixed, the following holds:
    \begin{itemize}
        \item If $y\mapsto g(x,y,z)$ is decreasing:
            Lower semicontinuity, convexity in $z$, and $h^k$ being decreasing in $y$ yield
            \begin{multline*}
                Y_s-\int_{s}^t h^k_u(Y_u,Z_u)du +\int_s^t Z_udW_u\geq \limsup_{n}\sum_{i=n}^{m_n} \alpha_i^n \left(Y^i_s-\int_{s}^{t} h^k_u\left(Y_u, Z^i_u \right)du +\int_{s}^{t} Z^i_u dW_u\right)\\
                \geq \limsup_n\sum_{i=n}^{m_n} \alpha_i^n \left(Y^i_s-\int_{s}^{t} h^k_u\left(Y_u^i, Z^i_u \right)du +\int_{s}^{t} Z^i_u dW_u\right).
            \end{multline*}
        \item If $y\mapsto g(x,y,z)$ is increasing: 
            Lower semicontinuity, convexity in $z$, the fact that $Y^n \to Y$ $P\otimes dt$-almost everywhere, the function $h^k$ being increasing in $y$, and $Y^n\leq Y^i$ for every $i=n,\ldots, m_n$, yield
            \begin{multline*}
                Y_s-\int_{s}^t h^k_u(Y_u,Z_u)du +\int_s^t Z_udW_u\geq \limsup_{n}\sum_{i=n}^{m_n} \alpha_i^n \left(Y^i_s-\int_{s}^{t} h^k_u\left(Y_u^n, Z^i_u \right)du +\int_{s}^{t} Z^i_u dW_u\right)\\
                \geq \limsup_n\sum_{i=n}^{m_n} \alpha_i^n \left(Y^i_s-\int_{s}^{t} h^k_u\left(Y_u^i, Z^i_u \right)du +\int_{s}^{t} Z^i_u dW_u\right).
            \end{multline*}
        \item If $(y,z)\mapsto g(x,y,z)$ is jointly convex: thereby $h^k$ is jointly convex too.
            Lower semicontinuity and joint convexity of $h^k$ yield
            \begin{multline*}
                Y_s-\int_{s}^t h^k_u(Y_u,Z_u)du +\int_s^t Z_udW_u\geq \limsup_n\sum_{i=n}^{m_n} \alpha_i^n \left(Y^i_s-\int_{s}^{t} h^k_u\left(Y_u^i, Z^i_u \right)du +\int_{s}^{t} Z^i_u dW_u\right).
            \end{multline*}
    \end{itemize}
    In all cases above, for every $n$ greater than $k$, it follows that $h^k(Y_u^i, Z^i_u)\leq h^i(Y_u^i, Z^i_u)$ for every $i=n,\ldots, m_n$.
    Hence
    \begin{multline*}
        \limsup_n\sum_{i=n}^{m_n} \alpha_i^n \left(Y^i_s-\int_{s}^{t} h^k_u\left(Y_u^i, Z^i_u \right)du +\int_{s}^{t} Z^i_u dW_u\right)\\
        \geq \limsup_n\sum_{i=n}^{m_n} \alpha_i^n \left(Y^i_s-\int_{s}^{t} h^i_u\left(Y_u^i, Z^i_u \right)du +\int_{s}^{t} Z^i_u dW_u\right)\geq\limsup_n\sum_{i=n}^{m_n} \alpha_i^n Y^i_t=Y_t
    \end{multline*}
    which, plugged into equation \eqref{eq:temp01}, yields
    \begin{equation*}
        Y_s-\int_{s}^t h_u(Y_u,Z_u)du +\int_s^t Z_udW_u\geq Y_t.
    \end{equation*}
    As $Y_T =\lim Y^n_T\geq \lim \xi^n=\xi$, this shows that $(Y,Z)\in \mathcal{A}(\xi, h)$.
    Having identified $(Y,Z)$ as a supersolution with terminal condition $\xi$ and driver $h$, this implies $\mathcal{E}_0(\xi, h)\leq Y_0$.
    Since $Y^n\leq \mathcal{E}(\xi,h)$ this completes the proof of $\mathcal{E}_0(\xi,h)=\lim \mathcal{E}_0(\xi^n, h^n)$.
    Particularly, an inspection of the arguments above yields that, whenever $\mathcal{E}_0(\xi,h)<\infty$, then $\mathcal{E}_t(\xi^n,h^n)$ increases monotonically to $\mathcal{E}_t(\xi, h)$ for every $t$.
    With this at hand, the monotone assertion \eqref{thm_stab2} follows readily by observing $h^n=g(X^n, \cdot,\cdot)$ for every $n$ as well as $\xi^n=\varphi(X^n)$.

    As for the first assertion \eqref{thm_stab1}, on the one hand, by definition of $h^n$ and $\xi^n$ for every $n$ it holds $h^n\leq g(X^n,\cdot,\cdot)$ and $\xi^n\leq \varphi(X_T^n)$.
    Hence $\mathcal{E}_0(\xi^n,h^n)\leq \mathcal{E}_0(X^n)$ for every $n$, showing that $\mathcal{E}_0(\xi, h)\leq \liminf \mathcal{E}_0(X^n)$.
    On the other hand, the lower semicontinuity of $\varphi$ implies $\varphi( X_T)\leq\xi $.
    Furthermore, since $g$ satisfies \ref{rec}, it holds $g(X,\cdot,\cdot)\leq h$.
    Combining the above we obtain $\mathcal{E}_0(X)\leq \mathcal{E}_0(\xi,h)\leq \liminf \mathcal{E}_0(X^n)$, thereby finishing the proof.
\end{proof}
As the preceding proof exhibits, the stability depends heavily on the generator $g$ satisfying \ref{rec}.
The following proposition shows that this assumption is indeed fulfilled in many circumstances.
The main part of its proof, being of convex analytical nature, is addressed in Appendix \ref{appendix01}.
\begin{proposition}\label{prop:lsc}
    A standard generator $g$ satisfies the assumption \ref{rec} in any of the following cases:
    \begin{enumerate}[label=\textit{(\roman*)}]
        \item\label{cond1} $g(x,y,z)=g_1(x)+g_2(y,z)$ with $g_1$ lower semi-continuous and $g_2$ a standard generator;
        \item\label{cond2} $g$ satisfies \ref{con1} and $f^\infty=g^\infty(x_n,\cdot,\cdot)$ for every $n$ where $f=\text{clco}\{g(x_n,\cdot,\cdot)\colon n\}$;
        \item\label{cond3} $g$ satisfies \ref{con1} and for every $\gamma$ the level set $\cup_n \{(y,z)\colon g(x_n,y,z)\leq \gamma\}$ is relatively compact;
        \item\label{cond4} $g$ satisfies \ref{mon} and for every $y$ and $\gamma$ the level set $\cup_n\{z\colon g(x_n,y,z)\leq \gamma\}$ is relatively compact.
    \end{enumerate}
    Cases \ref{cond2}--\ref{cond4} have to hold for every $(x_n)\subseteq \text{dom}g$.
\end{proposition}
\begin{proof}
    As for \ref{cond1}, due to \ref{con1} we have $\text{clco}\{g(x_k,\cdot,\cdot)\colon k\geq n\}=\inf_{k\geq n}g_1(x_k)+g_2$.
    Hence, since $g_1$ is lower semi-continuous, it holds
    \begin{equation*}
        \mathfrak{c}\text{-}\liminf g(x_n,\cdot,\cdot)=\liminf g_1(x_n)+g_2\geq g_1(x)+g_2=g(x,\cdot,\cdot).
    \end{equation*}
    The same argumentation is valid in the case where \ref{mon} is satisfied by considering the convex hull solely in $z$.
    
    The cases \ref{cond2} and \ref{cond4} are subjects of the Proposition \ref{lem:convergence} and \ref{lem:convergence2} in the Appendix \ref{appendix01}.
    Finally, a slight modification of Proposition \ref{lem:convergence2} in the jointly convex case yields \ref{cond3}.
\end{proof}
\begin{remark}
    Note that assumption \ref{cond2} is satisfied if $g(x,y,z)\geq h(y,z)$ for some lower semi-continuous and convex function such that $h^\infty=g^\infty(x,\cdot,\cdot)$ for every $x$.
    In particular if $h$ is coercive in which case \ref{cond3} also holds.
    Assumption \ref{cond4} is fulfilled if $g(x,y,z)\geq c(y)\abs{z}$ for some $c(y)>0$.
\end{remark}

We conclude this section by a further central property of forward backward minimal supersolutions, namely their locality with respect to the underlying forward process. 
\begin{proposition}\label{prop:locality}
    For $t\in[0,T]$ fixed, let $X^1,X^2$ be two forward processes and $A\in \mathcal{F}_t$.
    Define the forward process $X=X^11_{[0,t[}+(1_{A} X^1 +1_{A^c}X^2)1_{[t,T]}$ and suppose that $\mathcal{A}(X), \mathcal{A}(X^1),\mathcal{A}(X^2)\neq \emptyset$.
    Then it holds
    \begin{equation*}
        \mathcal{E}_s(X)=1_{A}\mathcal{E}_s(X^1)+1_{A^c}\mathcal{E}_s(X^2), \quad t\leq s\leq T.
    \end{equation*}
\end{proposition}
\begin{proof}
    Let us denote by
    \begin{equation}\label{eq_local1}
        \mathcal A_{t}(X) :=\crl{(Y,Z) \in\mathcal S_{|[t,T]}\times \mathcal L_{|[t,T]}\colon Z \text{ is admissible and \eqref{eq_original_tarpo} holds on }[t,T]}
    \end{equation}
    the set of supersolutions on $[t,T]$ and by $\mathcal{A}(X)_{|[t,T]}$ the restriction to $[t,T]$ of the elements of $\mathcal{A}(X)$.
    Clearly, $\mathcal{A}(X)_{|[t,T]}\subseteq \mathcal{A}_t(X)$, implying that
    \begin{equation*}
        I_s^{t}(X):=\essinf \left\{ Y_s\colon (Y,Z)\in \mathcal{A}_t(X) \right\}\leq \essinf \left\{ Y_s\colon (Y,Z)\in \mathcal{A}(X) \right\}=\mathcal{E}_s(X), \quad t\leq s\leq T.
    \end{equation*}
    Reversely, if $\mathcal{A}(X)\neq \emptyset$, then equality holds.
    Indeed, an application of Theorem \ref{thm:existunique} restricted to $[t,T]$ yields the existence of $\hat{Z} \in \mathcal{S}_{|[t,T]}$ such that $(I^t(X),\hat{Z})\in \mathcal{A}_t(X)$.
    For $(Y,Z) \in \mathcal{A}(X)$, it follows that $Y_t\geq \mathcal{E}_t(X)\geq I^{t}_t(X)$.
    Hence, by stability of supersolutions with respect to pasting, compare \citep[Lemma 3.1]{CSTP}, the pair defined by
    \begin{equation*}
        \tilde{Y}=1_{[0,t[}Y+1_{[t,T]}I^{t}(X)\quad\text{and}\quad\tilde{Z}=1_{[0,t]}Z+1_{]t,T]}\hat{Z}
    \end{equation*}
    belongs to $\mathcal{A}(X)$. 
    However, this implies $\tilde{Y}_s=I_s^{t}\geq \mathcal{E}_s(X)$ for $t\leq s\leq T$ and thus
    \begin{equation}\label{eq_I_equals_Ebis}
        I_t^{t}(X) = \mathcal{E}_t(X).
    \end{equation}
    With this at hand, under the assumption $\mathcal{A}_t(X^1),\mathcal{A}_t(X^2),\mathcal{A}_t(X)\neq \emptyset$ it is straightforward to check that $\mathcal{A}_t(X)=1_A\mathcal{A}_t(X^1)+1_{A^c}\mathcal{A}_t(X^2)$ since $A \in \mathcal{F}_t$ and therefore $I_s^t(X)=1_{A}I_s^t(X^1)+1_{A^c}I_s^t(X^2)$ for every $t\leq s\leq T$.
    In combination with \eqref{eq_I_equals_Ebis} the former yields
    \begin{equation*}
        \mathcal{E}_s(X)=1_A\mathcal{E}_s(X^1)+1_{A^c}\mathcal{E}_s(X^2), \quad t\leq s\leq T,
    \end{equation*}
    the proof is done.
\end{proof}

\section{Markovian minimal supersolutions}\label{sec4}
For the remainder, the forward process $X$ is given by the solution of the stochastic differential equation
\begin{equation*}\label{eq_SDE}
    X_t =X_0+\int_{0}^{t} \mu_u(X_u)du + \int_{0}^{t} \sigma_u(X_u)dW_u,
\end{equation*}
where $X_0 \in \mathbb{R}^n$ and $\mu\colon[0,T]\times \R^n \to \R^n$ and $\sigma \colon [0,T]\times \R^n \to \R^{n\times d}$ are jointly measurable functions satisfying the usual assumptions of SDE theory, namely
\begin{enumerate}[label=\textsc{(sde)},leftmargin=40pt]
    \item\label{sde} $\mu_\cdot(0)$ and $\sigma_\cdot(0)$ belong to $\mathcal L^2$; $\sigma$ and $\mu$ are uniformly Lipschitz and of linear growth in their second component.
\end{enumerate}
The goal of the current section is to show that in this case $\mathcal{E}_t(X)=u(t,X_t)$ where $u$ is a function defined on $[0,T]\times \mathbb{R}^n$.
To this end, given $t \in [0,T]$, we first define for every $\xi \in L^2(\mathcal{F}_t)$ the process $X^{t,\xi}$ as the unique solution of  
\begin{align}
    X^{t,\xi}_s &=\xi+\int_{t}^{s} \mu_u(X^{t,\xi}_u)du + \int_{t}^{s} \sigma_u(X^{t,\xi}_u)dW_u,&& t\leq s\leq T \nonumber\\
    X_s^{t,\xi}&=\xi-\int_s^t \mu_u(X^{t,\xi}_u)du-\int_{s}^t Z_u dW_u,&& 0\leq s\leq t.\label{eq_SDE_BSDE}
\end{align}
Notice that $X^{t,\xi}$ is well defined and uniquely determined. 
Indeed, it is the unique solution of an SDE with Lipschitz coefficients between $t$ and $T$ and initial value $\xi \in L^2(\mathcal{F}_t)$ and the unique solution of the Lipschitz BSDE with driver $\mu$ between $0$ and $t$ and terminal condition $\xi$.
It is furthermore continuous and adapted.
Uniqueness of these solutions in particular yields
\begin{equation*}
    X=X^{t,X_t}
\end{equation*}
and for every $\xi=\sum_{k=1}^n 1_{A_k} x_k$, where $(A_k)\subseteq \mathcal{F}_t$ is a partition, it holds
\begin{equation}\label{eq:stab}
    X^{t,\xi}_s=\sum 1_{A_k}X_s^{t,x_k},\quad t\leq s\leq T.
\end{equation}
Next, we need to consider the $t$-shifted problem.
More precisely, let $W^t:=W_{t+\cdot}-W_t$ be the Brownian motion on $[0,T-t]$ together with the corresponding filtration $\mathcal F^t_s :=\sigma(W^t_r : 0\le r\le s)$.
Accordingly, for each $x\in\R^n$ define $\tilde{X}^{t,x}$ as the solution of the stochastic differential equation
\begin{equation*}
    \tilde{X}^{t,x}_s=x+\int_0^s \mu_{t+u}(\tilde{X}^{t,x}_u)du+\int_{0}^{s}\sigma_{t+u}(\tilde{X}^{t,x}_u)dW_u^t, \quad 0\leq s\leq T-t.
\end{equation*}
Similarly, $t$-shifted supersolutions are those pairs $(Y,Z)\in \mathcal{S}(\mathcal{F}^t)\times \mathcal{L}(\mathcal{F}^t)$ such that
\begin{equation}\label{eq_tarpo_shift}
    Y_r-\int_r^s g_{t+u}(\tilde{X}^{t,x}_u,Y_u,Z_u)du+\int_r^s Z_{u}dW^t_{u}\geq Y_s\quad\mbox{and}\quad Y_{T-t}\geq \varphi\left(\tilde{X}_{T-t}^{t,x}\right).
\end{equation}
and we collect all $t$-shifted supersolutions on $[0,T-t]$ in the set
\begin{equation*}
    \tilde{\mathcal{A}}(\tilde{X}^{t,x}):= \left\{ (Y,Z) \in\mathcal S(\mathcal{F}^t)\times\mathcal L(\mathcal F^t) \colon \eqref{eq_tarpo_shift}\text{ holds and }\int ZdW^t \text{ is a supermartingale} \right\}.
\end{equation*}
Analogously, we denote by $\tilde{\mathcal{E}}(\tilde X^{t,x})$ the $t$-shifted minimal supersolution operator and define our candidate function $u\colon[0,T]\times \mathbb{R}^n\to [-\infty, \infty]$ by
\begin{equation*}\label{eq_def_u}
    u(t,x):=\inf \left\{Y_0 \colon (Y,Z)\in \tilde{\mathcal{A}}(\tilde{X}^{t,x})\right\}=\tilde{\mathcal{E}}_0(\tilde X^{t,x}).
\end{equation*}
The reader should keep in mind that for the sequel a ``tilde'' appearing in the notation of expressions always indicates a relation to the $t$-shifted problem on $[0,T-t]$ above.\\

The ensuing theorem provides the second contribution of this work by collecting important properties of $u$ and drawing the connection between the original problem, the $t$-shifted one and the function $u$.
\begin{theorem}\label{thm_main}
    We suppose that $g$ is a generator satisfying \ref{std} and \ref{rec}, $\mu$ and $\sigma$ satisfy \ref{sde}, and $\varphi$ is lower semicontinuous and linearly bounded from below.
    Then the following assertions hold true:
    \begin{enumerate}[label=\textit{(\roman*)}]
        \item $x\mapsto u(t,x)$ is lower semicontinuous, either identically $\infty$ or proper for every $t \in [0,T]$.
            If furthermore $g, \varphi, \mu$ and $\sigma$ are convex, then $x\mapsto u(t,x)$ is convex.
        \item If $\mathcal{A}(X^{t,x})\neq\emptyset$, then it holds
            \begin{equation*}\label{eq_equality_u_e}
                \mathcal{E}_t\left( X^{t,x} \right)=u(t,x).
            \end{equation*}
            In particular, $\mathcal{E}_t(X^{t,x})$ is a real number corresponding to the infimum of the $t$-shifted minimal solution problem.
        \item It holds
            \begin{equation*}\label{eq:equality}
                \mathcal{E}_t(X)\geq u(t,X_t)
            \end{equation*}
            with equality if $\mathcal{A}(X)\neq \emptyset$ and $x\mapsto u(t,x)$ is
            \begin{itemize}
                \item either continuous;
                \item or monotone and $X \geq C$ uniformly for some constant $C \in \mathbb{R}$.
            \end{itemize}
    \end{enumerate}
\end{theorem}

\begin{proof}
    For the remainder of the proof, we fix $t \in [0,T]$.
    \begin{enumerate}[label=\textit{Point (\roman*):}, fullwidth]
        \item For $x_n \to x$, up to a subsequence it holds $\lim \tilde{X}^{t, x_n}_{s}= \tilde{X}^{t,x}_s$ for every $s$ and $\inf_n \tilde{X}^{t,x_n}_T \in L^2$, both as a consequence of \citep[Theorem 2.4]{Touzi_BSDE}.
            Since $\varphi$ is lower semicontinuous and linearly bounded from below, it follows that $\inf_n \varphi(\tilde{X}^{t,x_n}_T)^-\in L^1$ and therefore the stability Proposition \ref{prop:stability} yields 
            \begin{equation*}
                \liminf u(t, x_n)=  \liminf\tilde{\mathcal{E}}_0(\tilde X^{t,x_n})\ge \tilde{\mathcal{E}}_0(\tilde X^{t,x})=u(t,x).
            \end{equation*}
            Finally, it holds that $\tilde{\mathcal{E}}_0(\tilde X^{t,x})\geq E[\varphi(\tilde{X}^{t,x}_T)]>-\infty$, by which we deduce that $u$ is either proper or uniformly equal to $\infty$.
            The proof of the convexity property goes along the lines of the argumentation in \citep[Proposition 3.3.(4)]{CSTP}.

        \item 
            First, let $X^{t,x}$ be defined as in \eqref{eq_SDE_BSDE}.
            In analogy to the proof of Proposition \ref{prop:locality} we obtain
            \begin{equation}\label{eq_I_equals_E}
            \mathcal{E}_t(X^{t,x}) = I_t^{t,x} 
            \end{equation}
            where $I^{t,x}_t=\essinf\{ Y_t: (Y,Z) \in \mathcal{A}_t(X^{t,x})\}$ and $\mathcal A_{t}(X^{t,x})$ is defined analogously to \eqref{eq_local1}.
            It remains to show the equality $I^{t,x}_t=u(t,x)$.
            In other terms, we need to establish the relation between the set $\mathcal{A}_t(X^{t,x})$ of supersolutions between $[t,T]$ with forward process $X^{t,x}$ and the set $\tilde{\mathcal{A}}(\tilde{X}^{t,x})$ of $t$-shifted supersolutions on $[0,T-t]$ with forward process $\tilde{X}^{t,x}$.
            Clearly, for every $(Y,Z)\in \tilde{\mathcal{A}}(\tilde{X}^{t,x})$, the observation $X^{t,x}_s=\tilde{X}^{t,x}_{s-t}$ implies that $(\bar{Y},\bar{Z}):=(Y_{\cdot-t},Z_{\cdot -t})\in \mathcal{A}_t(X^{t,x})$, showing in turn that $I^{t,x}_t\leq u(t,x)$.
            Together with \eqref{eq_I_equals_E} this implies $\mathcal E_t(X^{t,x}) \leq u(t,x)$.
            Reciprocally, since $\mathcal{A}(X^{t,x})$ is non-empty, so is $\mathcal{A}_t(X^{t,x})$ and thus there exists a control $Z^{t,x}$ corresponding to the $[t,T]$-minimal supersolution $I^{t,x}$. 
            Observe that for almost all $\bar\omega\in \Omega$ 
            \begin{equation*}
                \omega\mapsto  (Y^{\bar\omega}_s,Z^{\bar\omega}_s):= \brak{I_{s+t}^{t,x}(\bar\omega\otimes_{t}\omega), Z^{t,x}_{s+t}(\bar\omega\otimes_{t}\omega) }  \qquad s\in[0,T-t]
            \end{equation*}
            is a $t$-shifted supersolution with forward process $\tilde{X}^{t,x}$, that is, an element of $\tilde{\mathcal{A}}(\tilde{X}^{t,x})$.
            Indeed, it is measurable by definition and defines a pair of a \cadlag\, and a progressive process on $[0,T-t]$.
            In addition, this pair is adapted to $\mathcal F^t$.
            This follows from it being a functional of $\bar\omega\otimes_{t}\omega$ and thus by means of \eqref{eq_omega_concatenation} of $(\omega_{t+s}-\omega_t)_{s\in[0,T-t]}$.
            The fact that it satisfies \eqref{eq_tarpo_shift} follows from $\tilde{X}^{t,x}_{s}=X^{t,x}_{t+s}$ and the generator $g$ not depending on $\omega$.
            Hence, $(Y^{\bar\omega}_s,Z^{\bar\omega}_s)_{s\in[0,T-t]} \in \tilde{\mathcal{A}}(\tilde{X}^{t,x})$ and therefore, for almost all $\bar\omega\in \Omega$, it holds $Y^{\bar\omega}_0 \ge  \tilde{\mathcal{E}}_0(\tilde X^{t,x}) = u(t,x)$.
            Using the definition of $Y^{\bar\omega}$ in combination with \eqref{eq_I_equals_E} we obtain
            \begin{equation*}
                \mathcal E_t(X^{t,x}) = I^{t,x}_t \geq u(t,x) \quad P \text{-almost surely},
            \end{equation*}
            proving Point $(ii)$.

        \item 
            The inequality $\mathcal{E}_t(X)\geq u(t,X_t)$ is obtained by the path-wise argumentation of the previous point.
            Suppose now that $x\mapsto u(t,x)$ is continuous or increasing.
            Since $x\mapsto u(t,x)$ is lower semi-continuous, if 
            \begin{itemize}
                \item it is continuous, for every sequence of random variables $(X^n_t)\subseteq L^2(\mathcal{F}_t)$ converging to $X_t$, it holds $\lim u(t,X_t^n)=u(t,X_t)$.
                    In this case, we approximate $X_t$ by step functions $X^n_t\to X_t$ where for each $n$ we have $X^n_t= \sum_{k=1}^n1_{A^n_k} x^n_k$.
                \item it is monotone, for every increasing sequence of random variables $(X^n_t)\subseteq L^2(\mathcal{F}_t)$ converging to $X_t$, it holds $\lim u(t,X_t^n)=u(t,X_t)$.
                    In this case, since $X \geq C$ uniformly, we approximate $X_t$ from below by step functions, that is $X^n_t\nearrow X_t$ where for each $n$ we have $X^n_t= \sum_{k=1}^n1_{A^n_k} x^n_k$.
            \end{itemize}
            Using $(X^n_t)$ we define the family of terminal values $(X^n_T)$ by 
            \begin{equation*} 
                X^n_T := X^{t,\sum_{k=1}^n1_{A^n_k} x^n_k}_T
            \end{equation*} 
            which, by means of \eqref{eq:stab}, satisfy
            \begin{equation*} 
                X^{t,\sum_{k=1}^n1_{A^n_k} x^n_k}_s = \sum_{k=1}^n1_{A^n_k}  X^{t,x^n_k}_s,\quad t\leq s\leq T.
            \end{equation*}
            It clearly holds $X^n_s\to X_s$ for every $s\geq t$ and in the case of monotonicity, $X^n_t\nearrow X_t$.
            The function $x\mapsto u(t,x)$ being either increasing or continuous yields
            \begin{equation}\label{eq01}
                \liminf u(t,X^n_t)=\lim u(t,X_t^n)=u(t,X_t).
            \end{equation}
            Furthermore, by locality of $\mathcal{E}$, see Proposition \ref{prop:locality}, we have
            \begin{equation}\label{eq02}
                \mathcal{E}_t(X^n)=\sum 1_{A_k^n} \mathcal{E}_t(X^{t,x_k^n})=\sum 1_{A_k^n} u(t,x_k^n)=u(t, X^n_t).
            \end{equation}
            Finally, the stability result of Theorem \ref{prop:stability} together with relations \eqref{eq01} and \eqref{eq02} yields
            \begin{equation*}
                \mathcal{E}_t(X)\leq \liminf \mathcal{E}_t(X^n)=\liminf u(t,X^n_t)=u(t,X_t),
            \end{equation*}
            showing the reverse inequality and thereby completing the proof.
    \end{enumerate} 
\end{proof}

\section{Viscosity supersolutions}\label{sec5}

The last relation of Theorem \ref{thm_main}, namely $\mathcal E_t(X)=u(t,X_t)$, holds in the special cases of monotonicity or continuity.
The current and final section shows that it is also valid as soon as $g$ is jointly convex and even more, in this case the minimal supersolution can be interpreted as a viscosity supersolution of a corresponding PDE.

To begin with, following the notations and definitions in \cite{lions1992}, \cite{pardoux2014} and \cite{Touzi_BSDE}, we consider semilinear parabolic PDEs with terminal conditions of the form
\begin{equation}
        -\partial_t v(t,x)-F(t,x,v(t,x), Dv(t,x), D^2 v(t,x))=0,\quad \text{and}\quad v(T,x)=\varphi(x)
    \label{eq:pde01}
\end{equation}
with $v:[0,T]\times \mathbb{R}^d\to \mathbb{R}$, $\varphi: \mathbb{R}^d\to \mathbb{R}$ and $F:[0,T]\times \mathbb{R}^d\times \mathbb{R}\times \mathbb{R}^d\times \mathcal{S}(d)\to \mathbb{R}$.
Here, $\mathcal{S}(d)$ denotes the set of symmetric $d\times d$ matrices, while $F$ is supposed to be lower semicontinuous.
Further, $Dv$ and $D^2v$ corresponds to the gradient vector and matrix of second partial derivatives of $v$, respectively.
In the case under consideration $F$ is of the form
\begin{equation*}
    F(t,x,v,Dv,D^2v)=\mu_t(x)Dv+\text{tr}\left(\frac{1}{2}\sigma^2_t(x)D^2v\right)+g_t(x,v,\sigma_t(x)Dv).
\end{equation*}
Note that as $\sigma_t(x)$ is positive semi-definite, $F$ is degenerate elliptic.
\begin{definition}
    A viscosity supersolution of \eqref{eq:pde01} is a lower semicontinuous function $u:[0,T]\times \mathbb{R}^d\to \mathbb{R}$ such that 
    \begin{equation*}
        -a-F(t,x,u(t,x),p,M)\geq 0 \quad \text{for all }(t,x)\in [0,T]\times \mathbb{R}^d\text{ and }(a,p,M)\in \mathcal{P}^{-(1,2)}u(t,x)
    \end{equation*}
    where $\mathcal{P}^{-(1,2)}u(t,x)$ are the semi-jets of $u$ at $(t,x)$, that is those $(a,p,M)\in \mathbb{R}\times \mathbb{R}^d\times \mathcal{S}(d)$ satisfying
    \begin{equation*}
        u(t^\prime,x^\prime)\geq u(t,x)+a(t^\prime-t)+\langle p,x^\prime-x\rangle +\frac{1}{2}\langle M(x^\prime-x),x^\prime-x\rangle +o\left( \abs{t^\prime-t}+\abs{x^\prime-x} \right)
    \end{equation*}
    for every $(t^\prime,x^\prime)\in [0,T]\times \mathbb{R}^d$.
\end{definition}

\begin{theorem}\label{thm_main2}
    Assume that the assumptions of Theorem \ref{thm_main} are fulfilled and $g$ is convex.
    If in addition $\varphi$ is bounded from below, that is $\varphi\ge C$ for some $C\in\R$, and $\mathcal{A}(X)\neq\emptyset$, then it holds
    \begin{equation}\label{eq_equality_u_e2}
        \mathcal{E}_t(X)=u(t,X_t), \quad t \in [0,T].
    \end{equation}
    Furthermore, $u$ is the unique minimal\footnote{In the sense that for any other viscosity supersolution $v$ of the PDE \eqref{eq:pde01}, it holds $v\geq u$.} lower semicontinuous viscosity supersolution of the PDE \eqref{eq:pde01}.
\end{theorem}

\begin{proof}
    Note that if $\mathcal{A}(X)\neq \emptyset$, then $g$ is proper.
    As in \cite{DrapeauTangpi}, for each $n$ define
    \begin{equation*}
        g^n(x,y,z):= \sup_{\abs{\alpha}\vee\abs{\beta}\vee \abs{\gamma}\leq n} \crl{\alpha x + \beta y + \gamma z - g^*(\alpha,\beta,\gamma)}\quad \text{and}\quad \varphi^n(x)=\varphi(x)\wedge n
    \end{equation*}
    where $g^*$ is the convex conjugate of $g$.
    By Fenchel-Moreau, the sequence $(g^n)$ converges pointwise from below to $g$, while each $g^n$ is of linear growth.
    Being in addition convex, each $g^n$ is also Lipschitz continuous.
    Analogously to Section \ref{sec3}, we define $\mathcal E^n(X)$ as the minimal supersolution of the FBSDE with generator $g^n$, forward process $X$ and terminal function $\varphi^n$.
    As $g^n$ is Lipschitz and $\varphi^n$ is bounded, it follows from \citep[Remark 3.6]{DrapeauTangpi} that the minimal supersolution $\mathcal E^n(X)$ corresponds to the unique solution of the Lipschitz BSDE with generator $g^n$ and terminal condition $\varphi^n(X_T)$.
    
    Hence, a well-established result connecting Lipschitz BSDEs and semilinear PDEs, compare for instance \citep[Proposition 10.8]{Touzi_BSDE}, yields $u_n:[0,T]\times\R^d\to\R$ such that 
    \begin{equation}\label{equality_lipschitz}
        \mathcal E^n_t(X) = u_n(t,X_t).
    \end{equation}
    where $u_n$ is a continuous solution of the PDE \eqref{eq:pde01} with $F^n$ and $\varphi^n$ instead of $F$ and $\varphi$ respectively.
    Note that in addition, for each $t\in[0,T]$ the function $u_n(t,\cdot)$ corresponds exactly to the $t$-shifted problem with generator $g^n$ used in the proof of Theorem \ref{thm_main}.
    More precisely,
    \begin{equation*}\label{equality_lipschitz2}
        u_n(t,x)=\mathcal E^{n}_t(X^{t,x})=\tilde{\mathcal{E}}^n_0(\tilde X^{t,x})
    \end{equation*}
    with the notation analogous to above and $n$ indicating of course that $g^n$ is considered instead of $g$.
    Using the stability property of minimal supersolutions with respect to increasing drivers, see \citep[Theorem 4.14]{CSTP}, slightly adapted to in addition having increasing terminal conditions, it follows from $\mathcal{E}_0(X)<\infty$ that
    \begin{equation*}
        \mathcal{E}^n_t(X)\nearrow \mathcal{E}_t(X).
    \end{equation*}
    On the other hand, by the same argumentation for the shifted problem we deduce that
    \begin{equation*}
        u_n(t,x)\nearrow u(t,x),
    \end{equation*}
    pointwise which, together with \eqref{equality_lipschitz}, yields the desired relation \eqref{eq_equality_u_e2}.

    We are left to show that $u$ is a lower semicontinuous viscosity supersolution of the PDE \eqref{eq:pde01}.
    By means of \citep[Remark 6.3]{lions1992} it follows that
    \begin{equation*}
        u_{\ast}(t,x):=\liminf_{(n,t^\prime,x^\prime )\to (\infty, t,x)}u_n(t^\prime,x^\prime)
    \end{equation*}
    is a lower semicontinuous viscosity supersolution of \eqref{eq:pde01} with 
    \begin{equation*}
        F_{\ast}(t,x,u,p,M)=\liminf_{(n,t^\prime, x^\prime, u^\prime, p^\prime, M^\prime )\to (\infty, t,x,u,p,M)}F^n(t^\prime,x^\prime,u^\prime,p^\prime,M^\prime)
    \end{equation*}
    instead of $F$.
    However, from $g^n \nearrow g$ it follows that $F^n\nearrow F$.
    Since in addition $u_n \nearrow u$, Lemma \ref{lemma_cosima} below implies that $u_\ast=u$ and $F_\ast=F$, showing the existence.

    Let us finish the proof by showing the minimality of $u$.
    Let then $v$ be a lower semi-continuous viscosity supersolution of the PDE \ref{eq:pde01}.
    Since $F^n\leq F$ and $\varphi^n\leq \varphi$, it follows that $v$ is in particular a lower semi-continuous viscosity supersolution of the PDE \ref{eq:pde01} with $F^n$ and $\varphi^n$ instead of $F$ and $\varphi$ for every $n$.
    However, in this Lipschitz case, $u^n$ is in particular the unique lower semi-continuous viscosity supersolution of the PDE \ref{eq:pde01} with $F^n$ and $\varphi^n$.
    In particular, it follows that $v\geq u^n$ for every $n$.
    We thus deduce that $v\geq \sup_n u^n=u$, completing the proof.
\end{proof}
\begin{lemma}\label{lemma_cosima}
    Let $(h^n)$ be an increasing sequence of real valued continuous functions on $\mathcal O$ where $\mathcal O$ is a metric space.
    Then, for $h := sup_n h^n$ it holds that
    \begin{equation*}
        h(z) = h_*(z):=\liminf_{(n,z^\prime)\to(\infty,z)}h^n(z^\prime), \quad z \in \mathcal O.
    \end{equation*}
\end{lemma}
\begin{proof}
    Fix some $z\in\mathcal O$.
    By definition of the limes inferior we may pass to a subsequence, denoted by $(n,z_n)$, satisfying $\lim_n h^n(z_n)=h_*(z)$.
    For a fixed $k$, the sequence being increasing implies that $h^n(z_n)\geq h^k(z_n)$ for all $n$ sufficiently large.
    Combining the former with the continuity of $h^k$ yields
    \begin{equation*}
        h_*(z)\geq \lim_n h^k(z_n) = h^k(z), \quad \text{for all } k,
    \end{equation*}
    implying in turn that $h_*(t,z)\ge h(t,z)$.
    Conversely, for every $\varepsilon>0$ there exists $k$ such that for all $n\ge k$ it holds
    \begin{equation*}
        h_*(z)\leq \inf_{m\geq k}\inf_{\substack{z^\prime\neq z\\ d(z,z^\prime)\leq 1/n}} h^m(z^\prime) + \varepsilon\leq \inf_{\substack{z^\prime \neq z\\d(z,z^\prime)\le 1/n}} h^{k}(z^\prime) + \varepsilon.
    \end{equation*}
    By sending $n$ to infinity and subsequently using the continuity of $h^k$ as well as the definition of $h$ the above yields $h_*(z)\le h^{k}(z)+\varepsilon\le h(z)+\varepsilon$.
    As $\varepsilon$ was arbitrary, this finishes the proof. 
\end{proof}

\begin{appendix}
\section{Epi-convergence: technical results}\label{appendix01}
Throughout, let $X,Y,Z$ denote three finite dimensional euclidean real vector spaces.
We denote by $\text{cl}(C)$ and $\text{clco}(C)$ the closure and closure of the convex hull of a set $C$, respectively.
For a sequence of sets $(C^n)$, we define the \emph{Painlev\'e-Kuratowski} limit superior and the \emph{Closed Convex} limit superior by
\begin{equation*}
    \mathfrak{e}\text{-}\limsup C^n=\cap_n\text{cl}\left( \cup_{k\geq n} C^k\right)\quad\text{and}\quad\mathfrak{c}\text{-}\limsup C^n=\cap_n\text{clco}\left( \cup_{k\geq n} C^k\right),
\end{equation*}
respectively, see \citep[Chapter 4]{rockafellar02} and \citep{loehne2006}.
For a sequence $(f^n)$ of functions, we define $\mathfrak{e}\text{-}\liminf f^n$ or $\mathfrak{c}\text{-}\liminf f^n$ as the function the epigraph of which corresponds to the Painlev\'e-Kuratowsky or Closed-Convex limit superior of the epigraphs of $(f^n)$, respectively, see \citep[Chapter 7, Section B]{rockafellar02}.
In other terms,
\begin{equation*}
    \mathfrak{e}\text{-}\liminf f^n=\sup_n \text{cl}\left\{ f^k\colon k\geq n \right\}\quad \text{and}\quad \mathfrak{c}\text{-}\liminf f^n=\sup_n \text{clco}\left\{ f^k \colon k\geq n \right\}
\end{equation*}
where $\text{cl}\{f^k\colon k\geq n\}$ and $\text{clco}\{f^k\colon k\geq n\}$ is the greatest lower semicontinuous minorant and greatest lower semicontinuous convex minorant of every $f^k$ for $k\geq n$, respectively.
Clearly, it holds
\begin{equation*}
    \mathfrak{c}\text{-}\liminf f^n \leq \text{clco}\left\{\mathfrak{e}\text{-}\liminf f^n\right\}.
\end{equation*}
We denote by $C^\infty:=\{ x \colon \lambda_n x_n \to x\text{ for some }(x_n)\subseteq C \text{ and }\lambda_n \downarrow 0\}$ the \emph{horizon cone} of a set $C$.
Given a proper closed convex function $f$, we denote by $f^\infty$ its \emph{horizon function}, that is the function the epigraph of which corresponds to the horizon cone of the epigraph of $f$.

\begin{proposition}\label{lem:convergence}
    Let $f:X\times Z\to ]-\infty,\infty]$ be a proper lower semicontinuous function that is convex in $z$.
    Let $(x_n)\subseteq X$ with $x_n\to x$ and denote $f^n:=f(x_n,\cdot)$ and $h:=\textnormal{clco}\{f^n\colon n\}$.
    Suppose further that $(f^n)^\infty=h^\infty$.
    Then it holds
    \begin{equation*}\label{eq:central}
        f(x,z)\leq \text{clco}\left\{\mathfrak{e}\text{-}\liminf f^n\right\}(z)=\mathfrak{c}\text{-}\liminf f^n(z), \quad z \in Z.
    \end{equation*}
\end{proposition}
\begin{proof}
    If $f^n \equiv \infty$ except for finitely many $n$, then the inequality is trivially satisfied.
    Without loss of generality we may thus assume $f^n$ to be proper for every $n\in\N$.
    By lower semicontinuity of $f$ and \citep[Proposition 7.2]{rockafellar02} it follows that
    \begin{align*}
        \mathfrak{e}\text{-}\liminf f^n(z) & = \min \left\{ \alpha \in \mathbb{R}\colon \liminf f(x_n,z_n)=\alpha \text{ for some } z_n\to z\right\}\geq f(x,z),
    \end{align*}
    and since $f$ is lower semicontinuous and convex in $z$, we deduce
    \begin{equation*}
        f(x, z)\leq \text{clco}\{ \mathfrak{e}\text{-}\liminf f^n\}(z), \quad z \in Z.
    \end{equation*}
    
    Let now $C^n=\text{epi}f^n$.
    By assumption, $C^n$ is non-empty, closed and convex for every $n\in\N$.
    Furthermore, as $C:=\text{epi}(h)=\text{clco}(\cup_{n}C^n)$, it holds that $(C^n)^\infty=C^\infty$ for every $n$.
    Since horizon and recession cones coincide in finite dimensions for non-empty closed and convex sets, \citep[See][Theorem 8.2]{rockafellar01}, the conditions of \citep[Theorem 4.4]{loehne2006} are fulfilled and therefore we obtain $\text{clco}(\mathfrak{e}\text{-}\limsup C^n) = \mathfrak{c}\text{-}\limsup C^n$.
    This in turn implies 
    \begin{equation*}
        \text{clco}\left\{\mathfrak{e}\text{-}\liminf f^n\right\}(z)=\mathfrak{c}\text{-}\liminf f^n(z), \quad z \in Z,
    \end{equation*}
    finishing the proof.
\end{proof}

In the following, we consider the convex hull only with respect to certain dimensions which notation-wise is stressed by means of an index.
For instance, the convex hull in the second variable $z$ of a set $C\subseteq Y\times Z$  is denoted by $\text{co}_z(C)$.
\begin{proposition}\label{lem:convergence2}
    Let $f:X\times Y\times Z\to ]-\infty,\infty]$ be a proper lower semicontinuous function that is convex in $z$ and monotone in $y$.
    Suppose that for every bounded sequence $(x_n)\subseteq X$, $y \in Y$ and $\gamma \in \mathbb{R}$ the set $\cup_n \{z\colon f(x_n,y,z)\leq \gamma\}$ is contained in a compact set.
    Then, denoting $f^n:=f(x_n,\cdot)$, for every $(x_n)\subseteq X$ with $x_n \to x$  it holds
    \begin{equation*}
        f\left( x,y,z \right)\leq \textnormal{cl}\textnormal{co}_z\left\{ \mathfrak{e}\text{-}\liminf f^n \right\}(y,z)=\mathfrak{c}_z\text{-}\liminf f^n(y,z), \quad y,z \in Y\times Z.
    \end{equation*}
\end{proposition}
The argumentation is inspired by \citep[Lemma 1.1.9]{aubin2009}
\begin{proof}
    An argumentation analogous to the proof of Proposition \ref{lem:convergence} allows to assume that $f^n$ is proper for every $n$ and it holds
    \begin{equation*}
         f\left( x,y,z \right)\leq \text{cl}\text{co}_z\left\{ \mathfrak{e}\text{-}\liminf f^n \right\}(y,z), \quad y,z \in Y\times Z.
    \end{equation*}
    Furthermore, the relation
    \begin{equation}\label{eq:intermediate}
        \mathfrak{c}_z\text{-}\liminf f^n(y,z)\leq  \text{cl}\text{co}_z\left\{ \mathfrak{e}\text{-}\liminf f^n \right\}(y,z)\quad y,z \in Y\times Z.
    \end{equation}
    is naturally satisfied.
    Let $\gamma \in \mathbb{R}$ and define $C^n_\gamma:=\{(y,z)\colon f^n(y,z)\leq \gamma\}$.
    To show the reverse inequality in \eqref{eq:intermediate}, it is sufficient to show that $\text{clco}_{z}(\mathfrak{e}\text{-}\limsup C^n_\gamma)=\mathfrak{c}_z\text{-}\limsup C^n_\gamma$ for every $\gamma$.
    Let $(y,z) \in \mathfrak{c}_z\text{-}\limsup C^n_\gamma$ and with $d=\dim Z$ denote by $\Delta$ the $d+1$-dimensional simplex.
    By Caratheodory's Theorem, there exist sequences $(y_n), (z_n^i)_{i=1,\ldots,d+1},(\lambda_n)$ such that $(y_n,z_n^i) \in \cup_{k\geq n}C^k_\gamma $, $\lambda_n \in \Delta$, $y^n\to y$ and $\sum_i \lambda^i_n z_n^i\to z$.
    Up to a subsequence, we may assume that $\lambda_n \to \lambda \in \Delta$.
    Furthermore, for every $i$ it holds $(z_n^i)\subseteq \cup_n \{z\colon f(x_n,\tilde{y},z)\leq \gamma\}$ is contained in some compact set, since $(x_n)\subseteq X$ is bounded and where $\tilde{y}=\sup y_n$ or $\tilde{y}=\inf y_n$ depending on $f$ being increasing or decreasing in $y$.
    Hence, up to yet another subsequence, $z^i_n\to z^i$ holds for every $i$.
    In particular, $(y,z^i)\in \cap_n \text{cl}(\cup_{k\geq n}C^k_\gamma)$.
    Thus, $(y,z)=\lim (y_n, \sum_i \lambda^i_n z^i_n) =\sum \lambda^i (y,z^i) \in \text{cl}(\text{co}_z(\cap_n\text{cl}(\cup_{k\geq n}C^n_\gamma))= \text{cl}(\text{co}_{z}(\mathfrak{e}\text{-}\limsup C^n_\gamma))$ which ends the proof.
\end{proof}

\end{appendix}

\bibliographystyle{abbrvnat}
\bibliography{bibliography}
\end{document}